\documentclass[11pt, oneside]{amsart}
\usepackage{amsfonts, amstext, amsmath, amsthm, amscd, amssymb}
\usepackage{hyperref}
\usepackage{manfnt}
\usepackage{mathrsfs}
\usepackage{url}
\usepackage[all]{xypic}

\usepackage[paper=a4paper, text={138mm,208mm},centering]{geometry}

\usepackage{enumerate}
\usepackage{graphics, pinlabel, color}
\usepackage{comment}

\usepackage[all,graph]{xy}

\renewcommand{\setminus}{{\smallsetminus}}

\newcommand{\smfrac}[2]{\mbox{\footnotesize$\displaystyle\frac{#1}{#2}$}} 

\newcommand{\bp}{\begin{pmatrix}}
\newcommand{\ep}{\end{pmatrix}}
\newcommand{\be}{\begin{equation}}
\newcommand{\ee}{\end{equation}}
\newcommand{\ol}[1]{\overline{#1}}

\numberwithin{equation}{section}

\theoremstyle{plain}
\newtheorem{theorem}[equation]{Theorem}
\newtheorem{lemma}[equation]{Lemma}
\newtheorem{proposition}[equation]{Proposition}

\newtheorem{corollary}[equation]{Corollary}

\newtheorem*{claim*}{Claim}

\theoremstyle{definition}

\newtheorem{remark}[equation]{Remark}

\newtheorem{definition}[equation]{Definition}

\numberwithin{equation}{section}

 \newtheoremstyle{TheoremNum}
        {}{}              
        {\itshape}                      
        {}                              
        {\bfseries}                     
        {.}                             
        { }                             
        {\thmname{#1}\thmnote{ \bfseries #3}}
\theoremstyle{TheoremNum}

\def\Z{\mathbb Z}

\def\Q{\mathbb Q}
\def\C{\mathbb C}

\def\L{\Lambda}

\def\wt#1{\widetilde{#1}}

\def\p{\partial}

\def\sm{\setminus}

\def\a{\alpha}

\def\toiso{\xrightarrow{\simeq}}

\def\ll{\langle}

\def\rr{\rangle}

\def\bp{\begin{pmatrix}}
\def\ep{\end{pmatrix}}
\def\ba{\begin{array}}
\def\ea{\end{array}}
\def\bn{\begin{enumerate}}
\def\en{\end{enumerate}}

\def\zpx{\Z[\pi]}

\DeclareMathOperator\Ext{Ext}

\DeclareMathOperator\Hom{Hom}
\DeclareMathOperator\Aut{Aut}
\DeclareMathOperator\Id{Id}

\DeclareMathOperator\GL{GL}
\DeclareMathOperator\Bl{Bl}

\DeclareMathOperator\coker{coker}

\begin{document}

\title{Twisted Blanchfield pairings and symmetric chain complexes}

\author{Mark Powell}
\address{
D\'epartement de Math\'ematiques, Universit\'e du Qu\'ebec \`a Montr\'eal, QC, Canada}
\email{mark@cirget.ca}

\def\subjclassname{\textup{2010} Mathematics Subject Classification}
\expandafter\let\csname subjclassname@1991\endcsname=\subjclassname
\expandafter\let\csname subjclassname@2000\endcsname=\subjclassname

\begin{abstract}
We define the twisted Blanchfield pairing of  a symmetric triad of chain complexes over a group ring $\Z[\pi]$, together with a unitary representation of~$\pi$ over an Ore domain with involution.
 We prove that the pairing is sesquilinear, and we prove that it is hermitian and nonsingular under certain extra conditions.  A twisted Blanchfield pairing is then associated to a 3-manifold together with a decomposition of its boundary into two pieces and a unitary representation of its fundamental group.
\end{abstract}
\maketitle

\section{Introduction}

Let $(N;A,B)$ be an oriented 3-manifold with a decomposition $\partial N = A \cup B$ into codimension zero submanifolds, with $A \cap B$ a 1-submanifold of $\partial N$. Each of $A$, $B$ and $A \cap B$ can be empty.  Let~$R$ be a (not necessarily commutative) ring with involution.  Given a right $R$-module $P$, as described in Section~\ref{section:basic-chain-cx-constructions}, we can use the involution on $R$ to produce a left $R$-module $P^t$.   Let $\pi:=\pi_1(N)$, let $V$ be an $(R,\Z[\pi])$-bimodule, and let $\Theta \colon V \to V^*=\Hom_R(V,R)^t$ be an isomorphism of $(R,\Z[\pi])$-bimodules.  Let $S \subset R$ be a multiplicative subset with respect to which the Ore condition is satisfied. Denote the $S$-torsion submodule of an $R$-module $M$ by $TM$. For full details on this data see Section~\ref{section:algebraic-defn-TBF}.
The \emph{twisted Blanchfield pairing}
\[\Bl \colon TH_1(N,A;V) \times TH_1(N,B;V) \to S^{-1}R/R\]
is sesquilinear, and is by definition adjoint to the sequence of homomorphisms
\begin{equation}\label{eqn:sequence-of-HMs}
TH_1(N,A;V) \to TH^2(N,B;V)
\to \Hom_R(TH_1(N,B;V),S^{-1}R/R)^t\end{equation}
given by (the inverse of) Poincar\'{e}-Lefschetz duality, followed by a map induced by universal coefficients and a Bockstein homomorphism associated to the short exact sequence of coefficients $0 \to R \to S^{-1}R \to S^{-1}R/R \to 0$.  The isomorphism $\Theta$ is used to obtain Poincar\'{e}-Lefschetz duality with $V$ coefficients from Poincar\'{e} duality with $\Z[\pi]$-coefficients; details are given in Lemma~\ref{lem:isoduals} and Definition~\ref{defn:chain-level-blanchfield}.

As a basic example, let $N$ be a closed oriented $3$-manifold and let $\Z=R=V$, with the $\Z\pi$ action on $V=\Z$ given by the augmentation.  Then the twisted Blanchfield pairing is just the classical torsion linking pairing
\[TH_1(N;\Z) \times TH_1(N;\Z) \to \Q/\Z.\]
Another important special case is the
original Blanchfield pairing~\cite{Blanchfield:1957-1} on the free abelian cover of a closed oriented 3-manifold
\[\Bl \colon TH_1(N;\Z[G]) \times TH_1(N;\Z[G]) \to Q(G)/\Z[G],\]
where $G=H_1(N;\Z)/TH_1(N;\Z)$ and $Q(G)$ is the quotient field of $\Z[G]$.  For an oriented knot $K$ in $S^3$ and $X_K  :=S^3 \sm \nu K$ the exterior of $K$, we can take $G=\Z$ and obtain the Blanchfield pairing of the knot
\[\Bl \colon H_1(X_K;\Z[\Z]) \times H_1(X_K,\partial X_K;\Z[\Z]) \to Q(\Z)/\Z[\Z].\]
Note that $H_1(X_K;\Z[\Z]) \cong H_1(X_K,\partial X_K;\Z[\Z])$, and we will see in Section~\ref{section:recovering-classical-blanchfield} that we can replace $H_1(X_K,\partial X_K;\Z[\Z])$ with $H_1(X_K;\Z[\Z])$ to obtain a pairing on $H_1(X_K;\Z[\Z])$.


Twisted Blanchfield pairings, defined via one dimensional representations to non-commutative rings that factor through a non-abelian solvable quotient~$\pi/\pi^{(n)}$ of the fundamental group~$\pi$, have appeared in several recent papers, including~\cite{Burke:2014, Cha:2014-1, Cochran-Orr-Teichner:1999-1, Cochran-Harvey-Leidy:2009-1, Cochran-Harvey-Leidy:2011-03, Cochran-Harvey-Leidy:2011-02, DavisChris:2014, Franklin:2013, FLNP:2016, Leidy:2006-1, Nosaka-2016-1, Nosaka-2016-2}.  In particular, such twisted Blanchfield pairings of knots and links have been crucial to the study of higher order structures in the concordance of knots and links.

In Definition~\ref{defn:chain-level-blanchfield}, we give a purely algebraic and somewhat general definition of twisted Blanchfield pairings, in terms of A.\ Ranicki's symmetric Poincar\'{e} triads~\cite{Ranicki:1981-1}.  We will recall this theory in Section~\ref{section:Poincare-cxs-definitions}.
The goal of this paper is to use the algebraic viewpoint to give a systematic treatment of the key properties that a Blanchfield pairing might have, namely that it be sesquilinear, hermitian or nonsingular.  The hermitian property in particular does not seem to have been satisfactorily treated in the literature previously.  We collect the principal results of the paper into the next theorem.

\begin{theorem}\label{theorem:main}
Let $(N;A,B)$ be  an oriented 3-manifold with a decomposition $\partial N = A \cup B$ into codimension zero submanifolds, with $A \cap B$ a possibly empty 1-submanifold of $\partial N$.  Let $R$ be a (not necessarily commutative) ring with involution and let $S \subset R$ be a multiplicative subset with respect to which the Ore condition is satisfied. Let $\pi:=\pi_1(N)$ and let $V$ be an $(R,\Z[\pi])$-bimodule.
\begin{enumerate}[(a)]
  \item\label{item:mainthm-defined} The oriented $3$-manifold $N$ determines a symmetric Poincar\'{e} triad. Together with the bimodule $V$ and an isomorphism $\Theta \colon V \to V^*=\Hom_R(V,R)^t$ of $(R,\Z[\pi])$-bimodules, this gives rise to a twisted Blanchfield pairing
      \[\Bl \colon TH_1(N,A;V) \times TH_1(N,B;V) \to S^{-1}R/R\]
      that coincides with the pairing described in (\ref{eqn:sequence-of-HMs}).
  \item\label{item:mainthm-sesqui} The pairing $\Bl$ is sesquilinear, that is $\Bl(q\cdot x,p\cdot y) = q \Bl(x,y) \ol{p}$ for all $x \in TH_1(N,A;V)$, $y \in TH_1(N,B;V)$ and for all $p,q \in R$.
  \item\label{item:mainthm-nonsingular} Suppose in addition that
  \[\Ext^1_R(H_0(N,B;V),S^{-1}R/R) = 0.\]
  Then $\Bl$ is nonsingular; that is, the adjoint map
  \[TH_1(N,A;V) \to \Hom_{R}(TH_1(N,B;V),S^{-1}R/R)^t\]
  is an isomorphism of $R$-modules.
  \item\label{item:mainthm-hermitian}  Suppose that there exists a chain equivalence $\sigma \colon C_*(N,B;V) \simeq C_*(N,A;V)$ such that the diagram of $R$-module chain complexes
\[\xymatrix @R-0.5cm {& C_*(N;V) \ar[dl]_-{} \ar[dr]^-{} & \\  C_*(N,B;V) \ar[rr]^-{\simeq}_-{\sigma} && C_*(N,A;V)
  }\]
commutes up to homotopy.  Then $\sigma$ determines an identification \[TH_1(N,A;V) \toiso TH_1(N,B;V),\] with which we may define a pairing
  \[\Bl  \colon TH_1(N,A;V) \times TH_1(N,A;V) \to S^{-1}R/R.\]
This pairing is hermitian; that is $\Bl(x,y) = \ol{\Bl(y,x)}$ for all $x,y \in TH_1(N,A;V)$.
\end{enumerate}
\end{theorem}

The theorem will be proved in Section~\ref{section:algebraic-defn-TBF}.
We remark that the nonsingularity condition always holds when $R$ is a principal ideal domain and $S$ consists of the nonzero elements of $R$, since in that case $S^{-1}R/R$ is divisible and divisible modules over a principal ideal domain are injective.

Next, for the convenience of the reader, we give some situations in which the above result guarantees that we have a nonsingular and hermitian pairing. The simplest situation is for closed $3$-manifolds; when $A=B=\emptyset$, the condition for (\ref{item:mainthm-hermitian}) is trivially satisfied.

We remark that the data of an isomorphism $\Theta \colon V \to V^*$ is equivalent to an $R$-sesquilinear, nonsingular inner product on $V$ with respect to which the right $\Z[\pi]$-action on $V$ is given by a unitary representation $\a \colon \pi \to \Aut(V)$.  Further discussion can be found at the beginning of Section~\ref{section:algebraic-defn-TBF}.

Interesting examples of representations that record non-abelian data while remaining potentially computable arise as follows.  Consider a complex unitary representation $\beta \colon \pi \to U(d,\C)$ and a homomorphism $\phi \colon \pi \to F$, where $F$ is a finitely generated free abelian group. Consider $\C[F]$ as a ring with involution where the involution sends elements of $F$ to their inverse and acts by complex conjugation on the coefficients.
Then the tensor product representation \[\ba{rcl} \a = \phi \otimes \beta \colon \pi &\to & \Aut(\C[F] \otimes_{\C} \C^d) = \Aut(\C[F]^d) \\
 g &\mapsto & \phi(g) \cdot \beta(g) \ea\]
is also a unitary representation.

\begin{proposition}\label{propn-closed-3-mfld-intro}
 Let $N$ be a closed, oriented $3$-manifold and let $V$ be an $(R,\Z[\pi])$-bimodule with an inner product and the right $\Z[\pi]$ action determined by a unitary representation $\a \colon \pi \to \Aut(V)$ such that \[\Ext^1_R(H_0(N;V),S^{-1}R/R) = 0.\]  Then
 $\Bl  \colon TH_1(N;V) \times TH_1(N;V) \to S^{-1}R/R$ is nonsingular and hermitian.
\end{proposition}

Next, we give some easily verifiable criteria for a twisted Blanchfield pairing of a 3-manifold with boundary to be hermitian and nonsingular.

\begin{proposition}\label{propn-3-mfld-bdy-intro}
  Let $X$ be an oriented $3$-manifold with boundary and let $V$ be an $(R,\Z[\pi])$-bimodule with an inner product and the right $\Z[\pi]$ action determined by a unitary representation $\a \colon \pi \to \Aut(V)$, such that at least one of \[\Ext^1_R(H_0(X,\partial X; V),S^{-1}R/R) \text{ or }  \Ext^1_R(H_0(X; V),S^{-1}R/R)\] vanishes.
  Moreover, suppose that the quotient map $C_*(X;V) \to C_*(X,\partial X;V)$ induces an isomorphism $$TH_1(X;V) \toiso TH_1(X,\partial X;V).$$  Then using this identification, the Blanchfield pairing \[\Bl \colon TH_1(X;V) \times TH_1(X,\partial X;V) \to S^{-1}R/R\] induces a pairing $TH_1(X;V) \times TH_1(X;V) \to S^{-1}R/R$, that we also denote by $\Bl$.  The latter pairing is hermitian and nonsingular.
\end{proposition}

In Section~\ref{section:toroidal-boundary-divided} we consider the more complicated situation of a $3$-manifold with toroidal boundary, where each component of the boundary is decomposed into two annuli as $S^1 \times D^1 \cup_{S^1 \times S^0} S^1 \times D^1$.  One can also extract a nonsingular, hermitian pairing in this situation, with more care.

Further investigation of potential applications of twisted Blanchfield pairings remains an intriguing avenue for future research.  One hopes that the treatment provided in this paper, particularly the flexibility that we allow with regards to the boundary, will prove to be worthwhile.

\begin{remark}
  We note that the techniques of the paper also apply to higher odd dimensional manifolds, such as to the exterior of a high dimensional knot $S^{2k-1} \subset S^{2k+1}$. The definition of the Blanchfield pairing and the hermitian results extend in a straightforward fashion, except that for $k$ even the pairing is skew-hermitian.  The nonsingularity criterion of Proposition~\ref{prop:non-singular} is based on the universal coefficients spectral sequence, and thus requires more care to generalise both the statement and the proof.  
\end{remark}

Here is a summary of the remainder of the paper.  Section~\ref{section:Poincare-cxs-definitions} gives the necessary background and definitions on symmetric Poincar\'{e} chain complexes, pairs and triads.  Section~\ref{section:algebraic-defn-TBF} gives the algebraic definition of twisted Blanchfield pairings and gives conditions under which they are nonsingular and hermitian, proving Theorem~\ref{theorem:main}.  Section~\ref{section:recovering-classical-blanchfield} recovers the important special cases of the classical Blanchfield pairing of a knot and of the universal abelian cover of a $3$-manifold, as well as the torsion linking pairing on a branched cover of a knot.  Section~\ref{section:recovering-classical-blanchfield} also contains the proof of Proposition~\ref{propn-3-mfld-bdy-intro}. Section~\ref{section:toroidal-boundary-divided} deals with the case of toroidal boundary with each torus divided into two annuli.

\subsection*{Conventions.} Throughout the paper we assume that all $3$-manifolds are connected, compact and oriented, unless we say explicitly otherwise.

\subsection*{Funding} I gratefully acknowledge the support of an NSERC Discovery grant.

\subsection*{Acknowledgments.} I would like to thank Stefan Friedl, Matthias Nagel, Patrick Orson and Andrew Ranicki for several extremely helpful discussions and their invaluable comments and suggestions.  I also thank the referee for helping me to improve the exposition.
The basis for this paper derives from my 2011 University of Edinburgh PhD thesis, supervised by Andrew Ranicki, to whom I am just as grateful now as I was then.

\section{Symmetric Poincar\'{e} complexes}\label{section:Poincare-cxs-definitions}

An involution on a ring $R$ is an additive self map $a \mapsto \ol{a}$ with $\ol{a\cdot b} = \ol{b}\cdot \ol{a}$, $\ol{1}=1$ and $\ol{\ol{a}}=a$.
For example, given a group $\pi$  we will always view $\Z[\pi]$ as a ring equipped with the involution $\ol{\sum_{g\in \pi}n_gg}=\sum_{g\in \pi}n_gg^{-1}$.
The material of Section~\ref{section:Poincare-cxs-definitions} is due to Ranicki, primarily~\cite{Ranicki3}, and the reader looking for more details is referred to there.

\subsection{Basic chain complex constructions and conventions}\label{section:basic-chain-cx-constructions}

Let $R$ be ring with involution.  A left $R$-module $M$ becomes a right $R$-module using the involution, via the action $m \cdot a := \ol{a}m$ for $r\in R, m \in M$.  Denote this right module by $M^t$.  A similar statement holds with left and right switched.  We use the same notation $M^t$ in both instances.
Modules will be left modules by default.

\begin{definition}[Tensor chain complexes]\label{Defn:signsoontensor}
Given chain complexes $(C,d_C)$ and $(D,d_D)$ of finitely generated (f.g.) projective $R$-modules, with $C_r,D_r = 0$ for $r<0$, form the tensor product chain complex $C \otimes_{R} D$ with chain groups:
\[(C^t \otimes_R D)_n := \bigoplus_{p+q=n} \, C^t_p \otimes_R D_q.\]
 The boundary map
\[d_{\otimes} \colon (C^t \otimes_R D)_n \to (C^t \otimes_R D)_{n-1}\]
is given, for $x \otimes y \in C^t_p \otimes_R D_q \subseteq (C^t \otimes_R D)_n$, by
\[d_{\otimes}(x \otimes y) = x \otimes d_D(y) + (-1)^q d_C(x) \otimes y.\]
\end{definition}

\begin{definition}[$\Hom$ chain complexes]\label{defn:hom-chain-complex}
Define the complex $\Hom_R(C,D)$ by
\[\Hom_R(C,D)_n := \bigoplus_{q-p=n}\,\Hom_R(C_p,D_q)\]
with boundary map
\[d_{\Hom} \colon \Hom_R(C,D)_n \to \Hom_R(C,D)_{n-1}\]
given, for $g \colon C^p \to D_q$, by
\[d_{\Hom}(g) = d_D g + (-1)^q g d_C.\]
\end{definition}

\begin{definition}[Dual complex]\label{defn:dual-complex}
The dual complex $C^*$ is defined as a special case of Definition~\ref{defn:hom-chain-complex} with $D_0 = R$ as the only non-zero chain group, which is also an $R$-bimodule.  Explicitly we define $C^r := \Hom_R(C_r,R)^t$, with boundary map
$\delta=d^*_C \colon C^{r-1} \to C^{r}$
defined as
$\delta(g) = g \circ d_C.$
Using that $R$ is a bimodule over itself, the chain groups of $C^*$ are naturally right modules. But we use the involution to make them into left modules, as follows: for $f \in C^*$, let $(a \cdot f)(x) := f(x) \ol{a}$.

The chain complex $C^{-*}$ is defined to be
\[(C^{-*})_r = C^{-r};\;\; d_{C^{-*}} = (d_C)^*=\delta.\]
Also define the complex $C^{m-*}$ by:
\[(C^{m-*})_r = \Hom_R(C_{m-r},R)\]
with boundary maps
\[\partial^* \colon (C^{m-*})_{r+1} \to (C^{m-*})_{r}\]
given by
\[\partial^* = (-1)^{r+1}\delta.\]
\end{definition}

Define the dual of a cochain complex, that is the double dual, to be $C^{**}:=(C^{-*})^{-*}$.
The proof of the next proposition, that allows us to identify a chain complex with its double dual, is a straightforward verification.

\begin{proposition}[Double dual]
For a f.g.\ projective chain complex $C_*$, there is an isomorphism:
\[C_* \xrightarrow{\simeq} C^{**} ;\; x \mapsto (f \mapsto \ol{f(x)}).\]
\end{proposition}

\begin{definition}[Slant map]\label{defn:slant-map}
The slant map is the isomorphism:
\[\ba{rcl} \backslash \colon C^t \otimes_R C & \to & \Hom_R(C^{-*},C_*)\\
x \otimes y & \mapsto & \left(g \mapsto \overline{g(x)}y\right) \ea\]
\end{definition}

\begin{definition}[Transposition]
Let $C_*$ be a chain complex of projective left $R$-modules for a ring with involution~$R$.  Define the transposition map
\[\ba{rcl} T \colon C_p^t \otimes C_q &\to& C_q^t \otimes C_p\\
x \otimes y &\mapsto& (-1)^{pq} y \otimes x.\ea\]
This $T$ generates an action of $\Z_2$ on $C \otimes_{R} C$.
Also let $T$ denote the corresponding map on homomorphisms:
\[\ba{rcl} T \colon \Hom_R(C^p,C_q) &\to& \Hom_R(C^q,C_p)\\
\theta &\mapsto& (-1)^{pq} \theta^*. \ea\]
 \end{definition}

\begin{definition}[Algebraic mapping cone]\label{Defn:algmappingcone}
The algebraic mapping cone $\mathscr{C}(g)$ of a chain map $g \colon C \to D$ is the chain complex given by:
\[d_{\mathscr{C}(g)} = \left(\begin{array}{cc} d_D & (-1)^{r-1}g \\ 0 & d_C \end{array} \right) \colon \mathscr{C}(g)_r = D_r \oplus C_{r-1} \to \mathscr{C}(g)_{r-1} = D_{r-1} \oplus C_{r-2}.\]
The (co)homology groups $H_r(g)$  $(H^r(g))$ of a chain map $g$ are defined to be the (co)homology groups $H_r(\mathscr{C}(g))$ $(H^r(\mathscr{C}(g)))$ of the algebraic mapping cone.
\end{definition}

\subsection{Symmetric Poincar\'{e} complexes and closed manifolds}\label{section:symm-cxs-closed-mflds}

In this section we explain symmetric structures on chain complexes, following Ranicki~\cite{Ranicki3}.  In particular the chain complex of a manifold inherits a symmetric structure.
To begin, for simplicity, we take $M$ to be an $n$-dimensional \emph{closed} manifold with $\pi:= \pi_1(M)$ and universal cover $\wt{M}$.
Let
$\wt{\Delta} \colon \wt{M} \to \wt{M} \times \wt{M}$; $y \mapsto (y,y)$
be the diagonal map on the universal cover of $M$.  This map is $\pi$-equivariant, so we can take the quotient by the action of $\pi$ to obtain
\begin{equation}\label{topoldiagonal2}
\Delta \colon M \to \wt{M} \times_{\pi} \wt{M},
\end{equation}
where $\wt{M} \times_{\pi} \wt{M} := \wt{M} \times \wt{M}/(\{(x,y) \sim (gx,gy) \,|\,g \in \pi)$.
The symmetric structure arises from an algebraic version of this map.
The Eilenberg-Zilber theorem says that there is a natural chain equivalence $EZ  \colon C(\wt{M} \times \wt{M}) \simeq C(\wt{M}) \otimes_{\Z} C(\wt{M})$.
By abuse of notation, let
\[\wt{\Delta} \colon C(\wt{M}) \to C(\wt{M}) \otimes_{\Z} C(\wt{M})\]
be the composition of the induced map on chain complexes with $EZ$.
Take tensor product over $\Z[\pi]$ with $\Z$ of both the domain and codomain, to obtain:
\[\Delta_0 \colon C(M) \to C(\wt{M}) \otimes_{\Z[\pi]} C(\wt{M}).\]
The map $\Delta_0$ evaluated on the fundamental class $[M]$, composed with the slant map
yields \[\varphi_0:= \backslash \Delta_0([M]) \in \Hom_{\Z[\pi]}(C^{n-*}(\wt{M}),C_*(\wt{M})).\]
In the case $n=3$ we have a collection of $\Z[\pi]$-module homomorphisms of the form:
\[\xymatrix @C+1cm{
C^0 \ar[r]^{\partial^*_1} \ar[d]^{\varphi_0}& C^1 \ar[r]^{\partial^*_2} \ar[d]^{\varphi_0} & C^2 \ar[d]^{\varphi_0} \ar[r]^{\partial^*_3} & C^3 \ar[d]^{\varphi_0}\\
C_3 \ar[r]_{\partial_3} & C_2 \ar[r]_{\partial_2} & C_1 \ar[r]_{\partial_1} & C_0}\]
A symmetric structure also comprises higher chain homotopies $\varphi_s \colon C^r \to C_{n-r+s}$ which measure the failure of $\varphi_{s-1}$ to be symmetric on the chain level. We will introduce the higher symmetric structures next, using the higher diagonal approximation maps.

\begin{definition}\label{Defn:higherdiagonalmaps}
A \emph{chain diagonal approximation} is a chain map $\Delta_0 \colon C_* \to C_* \otimes C_*$, with a collection, for $i \geq 1$, of chain homotopies $\Delta_i \colon C_* \to C_* \otimes C_*$ between $\Delta_{i-1}$ and $T\Delta_{i-1}$.  That is, the $\Delta_i$ satisfy the relations:
\[\partial \Delta_i - (-1)^i\Delta_i\partial = \Delta_{i-1} + (-1)^iT\Delta_{i-1}. \]
\end{definition}

The following theorem of J.~Davis \cite{Davis} ensures the existence algebraically of the diagonal approximation for an acyclic chain complex.
For the ordinary singular chain complex of a space, one can use the Alexander-Whitney diagonal approximation, but for cellular or handle chain complexes with finitely generated chain groups, the next theorem can be reassuring.  In particular, most irreducible $3$-manifolds have contractible universal cover, whence their augmented $\Z[\pi]$-module chain complexes are acyclic.

\begin{theorem}[J.~Davis]\label{Thm:davisdiag}
Let $C = (C_i,\partial)_{0 \leq i \leq n}$ be a chain complex of free $\Z[\pi]$-modules in non-negative dimensions, with augmentation $\a \colon C_0 \to \Z$, such that the augmented chain complex is acyclic.  Then there exists a $\Z[\pi]$-module chain diagonal approximation $\Delta_i, i=0,1,\dots,n$ ($\Delta_i = 0$ for $i >n$), as in Definition \ref{Defn:higherdiagonalmaps}, satisfying:
\begin{enumerate}[(i)]
  \item For all $j$, $\Delta_j(C_i) \subset \bigoplus_{m\leq i,n\leq i}\,C_m\otimes C_n$.
  \item  $(\a \otimes 1) \circ \Delta_0 = 1$.
  \item $(1 \otimes \a) \circ \Delta_0 = 1$.
  \item For all $i$, for any $c \in C_i$, there is an $a \in C_i \otimes C_i$ such that
  $\Delta_i(c) - c \otimes c = a + (-1)^iTa$.
\end{enumerate}
Furthermore, any two choices of such maps are chain homotopic.
\end{theorem}

\begin{remark}
Higher diagonal maps are related to the Steenrod squares.  The cup product of $f \in H^i(C)$ and $g \in H^j(C)$ is given by
\[f\cup g = \Delta_0^*(f^t \otimes g) \in H^{i+j}(C),\]
where $f^t$ is the induced map on $C_i^t$. Suppose that $C$ is a chain complex over a ring $R$.  For a cohomology class $f \in H^r(C)$ we define
\[Sq^i(f) = \Delta_{r-i}^*(f^t \otimes f).\]
If $2R=0$ or if $i$ is odd this induces a map $Sq^i \colon H^r(C) \to H^{r+i}(C)$.
\end{remark}
\noindent Using the higher $\Delta_i$ we can define the entire symmetric structure on a chain complex.

\begin{definition}[Symmetric Poincar\'{e} chain complex]\label{Defn:Qgroups}
Let $W$ be the standard free $\Z[\Z_2]$-resolution of $\Z$ shown below.
\[\dots \to \Z[\Z_2] \xrightarrow{1+T} \Z[\Z_2] \xrightarrow{1-T}\Z[\Z_2] \xrightarrow{1+T}\Z[\Z_2] \xrightarrow{1-T} \Z[\Z_2].\]
Given a finitely generated projective chain complex $C_*$ over $R$, define the symmetric $Q$-groups to be:
\[Q^n(C) := H_n\big(\Hom_{\Z[\Z_2]}(W,C^t \otimes_R C)\big) \cong H_n\big(\Hom_{\Z[\Z_2]}(W,\Hom_{R}(C^{-*},C_*))\big)\]
An element $\varphi \in Q^n(C)$ can be represented by a collection of $R$-module homomorphisms
\[\{\varphi_s \in \Hom_R(C^{n-r+s},C_r)\,|\,r \in \Z, s \geq 0\}\]
such that:
\[d_C\varphi_s + (-1)^r \varphi_s\delta_C + (-1)^{n+s-1}(\varphi_{s-1}+(-1)^sT\varphi_{s-1}) = 0 \colon C^{n-r+s-1} \to C_r\]
where $\varphi_{-1} = 0$.  The signs which appear here arise from the signs in the boundary maps in Definitions \ref{Defn:signsoontensor} and \ref{defn:hom-chain-complex}.
A pair $(C_*,\varphi)$, with $\varphi \in Q^n(C)$, is called an $n$-dimensional symmetric $R$-module chain complex.  It is called an $n$-dimensional symmetric \emph{Poincar\'{e}} complex if the maps $\varphi_0 \colon C^{n-r} \to C_r$ form a chain equivalence.  In particular this implies that they induce isomorphisms (the cap products) on homology:
\[\varphi_0 \colon H^{n-r}(C) \xrightarrow{\simeq} H_r(C).\]
The symmetric structure is covariantly functorial with respect to chain maps, in that a chain map $f \colon C \to C'$ induces a map $f^\% \colon Q^n(C) \to Q^n(C')$ given by
\[f^\%(\varphi)_s = (f^t \otimes_R f)(\varphi_s) \in C'^t \otimes_R C';\;\text{or}\]
\[\varphi_s \mapsto f\varphi_sf^*.\]
A homotopy equivalence of $n$-dimensional symmetric complexes $f \colon (C,\varphi) \to (C',\varphi')$ is a chain equivalence $f \colon C \to C'$ such that $f^\%(\varphi) = \varphi' \in Q^n(C')$.
\end{definition}

\begin{remark}
The upper indices for induced maps $f^\%$ do not indicate contravariance. They are used to distinguish from quadratic structure, which we will not consider in this paper.  The interested reader could learn about quadratic structure in~\cite{Ranicki3}, for example.
\end{remark}

The symmetric construction, which is the process by which a manifold gives rise to a symmetric chain complex, as in the next proposition, appears in \cite[Part~II,~Proposition~2.1]{Ranicki3}.

\begin{proposition}
  An oriented $n$-dimensional manifold $M$ with $\pi_1(X) = \pi$ gives rise to a symmetric Poincar\'{e} chain complex \[(C := C_*(\wt{M}), \varphi_i := \backslash\Delta_i([M]) \in Q^n(C)),\]
unique up to homotopy equivalence.
\end{proposition}

\noindent A symmetric Poincar\'{e} complex need not arise from any closed manifold in general.

\subsection{Symmetric Poincar\'{e} pairs and manifolds with boundary}\label{symm-cxs-mfld-bdy}

Let $(X,\partial X)$ be an $(n+1)$-manifold with boundary, with $\pi:=\pi_1(X)$, universal cover $\wt{X} \xrightarrow{p} X$ and $\wt{\partial X}:= p^{-1}(\partial X)$.  Let $[X,\partial X] \in C_{n+1}(X;\Z)$ be a relative fundamental class, which maps under $C_{n+1}(X;\Z) \to C_{n+1}(X,\partial X;\Z)$ to a generator of $H_{n+1}(X,\partial X;\Z)$.  On the chain level, $d_C([X,\partial X]) = (-1)^{n+1} f([\partial X]) \in C_{n}(X)$, where $f$ is the chain level inclusion of the boundary into $X$, and $[\partial X]$ is the fundamental class of the boundary $\partial X$.  We have
\[d_{\otimes} \Delta_0([X,\partial X]) = \Delta_0 d_C ([X,\partial X]) = \Delta_0((-1)^{n+1} f([\partial X])),\]
from which the symmetric pair equations in  Definition~\ref{poincarepaireqns} below are derived.
For a manifold with boundary denote  the collection of maps given by $\backslash \Delta([X,\partial X])$ by $\delta \varphi$, and for the duality maps on the boundary $\backslash \Delta([\partial X])$ we use $\varphi$, since $\partial X$ is a closed manifold.
The pair
\[(f \colon C_*(\wt{\partial X}) \to C_*(\wt{X}), (\delta \varphi, \varphi))\]
is a symmetric Poincar\'{e} pair in the sense of the next definition~\cite[Part~II, Proposition~6.2]{Ranicki3}.  Recall that $\mathscr{C}$ denotes the algebraic mapping cone of Definition~\ref{Defn:algmappingcone}.

\begin{definition}\label{poincarepaireqns}
The \emph{relative $Q$-groups} of an $R$-module chain map $f \colon C \to D$ are defined to be:
\[Q^{n+1}(f) := H_{n+1}\big(\Hom_{\Z[\Z_2]}(W,\mathscr{C}(f^t \otimes_{R} f))\big).\]
An element $(\delta\varphi,\varphi) \in Q^{n+1}(f)$ can be represented by a collection:
\[\{(\delta\varphi_s,\varphi_s) \in (D^t \otimes_R D)_{n+s+1} \oplus (C^t \otimes_R C)_{n+s}\, |\, s \geq 0\}\]
such that:
\begin{eqnarray*}(d_{\otimes}(\delta\varphi_s) + (-1)^{n+s}(\delta\varphi_{s-1}+(-1)^sT\delta\varphi_{s-1}) + (-1)^nf\varphi_sf^*,\\  d_{\otimes}(\varphi_s) + (-1)^{n+s-1}(\varphi_{s-1}+(-1)^sT\varphi_{s-1})) = 0 \\ \in (D^t \otimes_{R} D)_{n+s} \oplus (C^t \otimes C)_{n+s-1}
\end{eqnarray*}
where as before $\delta\varphi_{-1} = 0 = \varphi_{-1}$.
A chain map $f\colon C \to D$ together with an element $(\delta\varphi,\varphi) \in Q^{n+1}(f)$ is called an \emph{$(n+1)$-dimensional symmetric pair}.  A chain map $f$ together with an element of $Q^{n+1}(f)$ is called an \emph{$(n+1)$-dimensional symmetric Poincar\'{e} pair} if the relative homology class $(\delta\varphi_0,\varphi_0) \in H_{n+1}(f^t \otimes_R f)$ induces isomorphisms
\[H^{n+1-r}(D,C) := H^{n+1-r}(f) \xrightarrow{\simeq} H_r(D) \;(0 \leq r \leq n+1).\]
For a symmetric Poincar\'{e} pair corresponding to an $(n+1)$-dimensional manifold with boundary, these are the isomorphisms of Poincar\'{e}-Lefschetz duality.
\end{definition}

\noindent The union construction is used to glue two symmetric pairs together.

\begin{definition}\label{Defn:unionconstruction}~\cite[Part I,~pages~134--6]{Ranicki3}
A \emph{symmetric cobordism} between symmetric complexes $(C,\varphi)$ and $(C',\varphi')$ is an $(n+1)$-dimensional symmetric Poincar\'{e} pair with boundary $(C \oplus C',\varphi \oplus -\varphi')$:
\[((f_C,f_{C'}) \colon C \oplus C' \to D,(\delta\varphi,\varphi \oplus -\varphi')\in Q^{n+1}((f_C,f_{C'}))).\]

We define the \emph{union} of two symmetric cobordisms:
\begin{align*}
  c  = ((f_C,f_{C'}) \colon C \oplus C' &\to D,(\delta\varphi,\varphi \oplus -\varphi')); \text{ and} \\
c' = ((f'_{C'},f'_{C''}) \colon C' \oplus C'' &\to D',(\delta\varphi',\varphi' \oplus -\varphi'')),
\end{align*}
to be the symmetric cobordism given by:
\[c \cup c' = ((f''_C,f''_{C''}) \colon C \oplus C'' \to D'',(\delta\varphi'',\varphi \oplus -\varphi'')),\]
where:
\[D''_r:= D_r \oplus C'_{r-1} \oplus D'_r;\]
\[d_{D''} = \left(\begin{array}{ccc} d_D & (-1)^{r-1}f_{C'} & 0 \\ 0 & d_{C'} & 0 \\ 0 & (-1)^{r-1}f'_{C'} & d_{D'}
\end{array}\right)\colon D''_r \to D''_{r-1};\]
\begin{align*}
  f''_{C} = \left(
              \begin{array}{c}
                f_C \\
                0 \\
                0
              \end{array}
            \right) &\colon C_r \to D''_r; \\
f''_{C'} = \left(
              \begin{array}{c}
                0 \\
                0 \\
                f'_{C''}
              \end{array}
            \right) &\colon C_r \to D''_r; \text{ and}
\end{align*}
\[\delta\varphi''_s = \left(
                        \begin{array}{ccc}
                          \delta\varphi_s & 0 & 0 \\
                          (-1)^{n-r}\varphi'_sf_{C'}^* & (-1)^{n-r+s+1}T\varphi'_{s-1} & 0 \\
                          0 & (-1)^sf'_{C'}\varphi'_s & \delta\varphi'_s \\
                        \end{array}
                      \right)\colon\]
\[(D'')^{n-r+s+1} = D^{n-r+s+1} \oplus (C')^{n-r+s} \oplus D'^{n-r+s+1} \to D''_r = D_r \oplus C'_{r-1} \oplus D'_r \;\;(s \geq 0).\]
We write:
\[(D'' = D \cup_{C'}D',\delta\varphi'' = \delta\varphi\cup_{\varphi'}\delta\varphi').\]
\end{definition}

The next proposition can be found in\cite[Proposition~6.7]{Ranicki-1978}. The proof consists of verifying that the preceding symmetric union construction gives a model for the push out in the homotopy category of symmetric pairs.

\begin{proposition}
  The symmetric cobordism of the geometric union of two cobordisms $(W;M,N) \cup_N (V;N,P)$ along a common boundary component $N$ is chain equivalent to the symmetric union $C_*(W) \cup_{C_*(N)} C_*(V)$ of their respective symmetric cobordisms.
\end{proposition}

\subsection{Symmetric Poincar\'{e} triads}\label{section:symm-Pe-triads}

Next, we give the definition of a symmetric Poincar\'{e} triad.  This is the algebraic version of a manifold with boundary where the boundary is split into two along a codimension one submanifold.

\begin{definition}\label{Defn:symmPoincaretriad}~\cite[Definition~20.26]{HDKT}
An \emph{$(n+2)$-dimensional symmetric (Poincar\'{e}) triad} is a triad of f.g.\ projective $R$-module chain complexes:
\[\xymatrix @C+0.5cm{\ar @{~>}[dr]^{g}
D \ar[r]^{j_A} \ar[d]_{j_B} & A \ar[d]^{i_A} \\ B \ar[r]_{i_B} & C
}\]
with chain maps $i_A$, $i_B$, $j_{A}$, $j_B$, a chain homotopy $g \colon i_A \circ j_A \simeq i_B \circ j_B$ (the square need only commute up to homotopy) and structure maps $(\Phi,\varphi^A,\varphi^B,\chi)$ such that:
$(D,\chi)$ is an $n$-dimensional symmetric (Poincar\'{e}) complex,
\[(j_B \colon D \to B,(\varphi^B,\chi)) \text{ and } (j_A \colon D \to A,(\varphi^A,\chi))\]
are $(n+1)$-dimensional symmetric (Poincar\'{e}) pairs, and
\[(e \colon A \cup_{D} B \to C,(\Phi,\varphi^A \cup_{\chi} -\varphi^B)) \]
is an $(n+2)$-dimensional symmetric (Poincar\'{e}) pair, where:
\[e = \left(\begin{array}{ccc} i_A \, , & (-1)^{r-1}g \, , & -i_B \end{array} \right) \colon (A)_r \oplus D_{r-1} \oplus (B)_r \to C_r.\]
\end{definition}

Now we suppose that \emph{all the maps of a symmetric (Poincar\'{e}) triad are split injections}.
This may be arranged up to chain equivalence for any symmetric triad, using the algebraic mapping cylinder construction~\cite[pages~80--1]{Ranicki:1981-1}.
Now the chain homotopy $g$ vanishes; that is, we require $i_A\circ j_A = i_B \circ j_B$ precisely, not just up to chain homotopy.

We may replace all mapping cones with cokernels.  Thus for a chain map $f \colon D \to C$ which is a split injection, let $C/D:= \coker(f \colon D \to C)$ and (re)define the homology groups $H_r(f) := H_r(C/D)$.  In the notation of Definition~\ref{Defn:symmPoincaretriad}, the quotient maps $q_A \colon C \to C/A$ and $q_B \colon C \to C/B$ and the symmetric structure maps induce chain maps \[(C/A)^{n+2-r} \xrightarrow{q_A^*} C^{n+2-r} \xrightarrow{\Phi_0} C_{r} \xrightarrow{q_B} (C/B)_r.\]
The central map $\Phi_0$ is not a chain map in general, but the composite $q_B\Phi_0q_A^*$ is a chain map.  This is not too hard to see, but just to be on the safe side a proof is provided as part of the proof of Lemma~\ref{lemma:hermitian-chain-homotopy}.
With $p_B \colon B \to B/D$ and $p_C \colon C \to C/(A \cup_D B)$ the quotient maps, by similar arguments we also have chain maps
\[\varphi^B_0  p_B^* \colon (B/D)^{n+1-r} \xrightarrow{p^*_B} B^{n+1-r} \xrightarrow{\varphi_0^B} B_{r}\]
and
\[\Phi_0  p_C^* \colon (C/(A\cup_D B))^{n+2-r} \xrightarrow{p^*_C} C^{n+2-r} \xrightarrow{\Phi_0} C_{r}.\]
Now suppose that the $(n+2)$-dimensional symmetric triad
\[\xymatrix {
D \ar[r]^{j_A} \ar[d]_{j_B} & A \ar[d]^{i_A} \\ B \ar[r]_{i_B} & C
}, (\Phi, \varphi^A,\varphi^B, \chi)\]
is Poincar\'{e} (since $g=0$, we omit it from the notation).  Then, by definition, $\varphi^B_0  p_B^*$ and $\Phi_0  p_C^*$ induce isomorphisms between cohomology and homology analogous to Poincar\'{e}-Lefschetz duality for manifolds with boundary.
It follows that the maps $q_B  \Phi_0 q_A^* \colon (C/A)^{n+2-r} \to (C/B)_r$ induce isomorphisms $H^{n+2-r}(i_A) \toiso H_r(i_B)$, by the five lemma applied to the following diagram:
\[\xymatrix @R+0.5cm @C-0.1cm {H^{n+1-r}(j_B) \ar[d]_{\cong}^{\varphi^B_0 p_B^*} \ar[r] & H^{n+2-r}(e) \ar[r] \ar[d]_{\cong}^{\Phi_0  p_C^*} & H^{n+2-r}(i_A) \ar[r] \ar[d]^{q_B\Phi_0q_A^*} & H^{n+2-r}(j_B) \ar[r] \ar[d]_{\cong}^{\varphi^B_0 p_B^*} & H^{n+3-r}(e) \ar[d]_{\Phi_0  p_C^*}^{\cong} \\
H_r(B) \ar[r] & H_r(C) \ar[r] & H_r(i_B) \ar[r] & H_{r-1}(B) \ar[r] & H_{r-1}(C).
}\]
The same applies with $A$ and $B$ switched.  By abuse of notation we will also refer to the map $q_B \Phi_0  q_A^* \colon (C/A)^{n+2-r} \to (C/B)_r$ as~$\Phi_0$.

The following proposition follows from three applications of the relative symmetric construction of \cite[Part~II,~Proposition~6.2]{Ranicki3}.
Given a connected space $X$ and a connected subspace $Y \subset X$, choose a path $\gamma_Y$ from the basepoint of $Y$ to the basepoint of $X$.  Produce a chain complex \[C_*(Y;\Z[\pi_1(X)]) := \Z[\pi_1(X)] \otimes_{\Z[\pi_1(Y)]} C_*(Y;\Z[\pi_1(Y)])\] as the chain complex of the pullback cover of $Y$, pulling back the universal cover of $Y$ with respect to the inclusion $Y \subset X$.    The map $\pi_1(Y) \to \pi_1(X)$ depends on the choice of path $\gamma_Y$.
Now let $Z := \bigsqcup Z_i \subset X$ be a disconnected subspace of $X$.  Choose a path $\gamma_{Z_i}$ from the basepoint of $X$ to the basepoint of $Z_i$ for each $i$, and thus obtain homomorphisms $\pi_1(Z_i) \to \pi_1(X)$. Define:
\[C_*(Z;\Z[\pi_1(X)]) := \bigoplus_i C_*(Z_i;\Z[\pi_1(X)]) = \bigoplus_i \Z[\pi_1(X)] \otimes_{\Z[\pi_1(Z_i)]} C_*(Z_i;\Z[\pi_1(Z_i)]).\]

\begin{proposition}\label{prop:triad-from-manifold}
  Let $N$ be an $n$-dimensional manifold with boundary $\partial N = A \cup_D B$, where $A,B \subset \partial N$ are codimension 0 submanifolds with $A \cap B = D$, a possibly empty codimension 1 submanifold of $\partial N$.  Let $\pi = \pi_1(N)$.  This manifold triad determines an $n$-dimensional $\zpx$-module symmetric Poincar\'{e} triad as in Definition~\ref{Defn:symmPoincaretriad}
with $C = C_*(N;\zpx)$ and $D$, $A$ and $B$ $\zpx$-modules chain complexes with the same notation as their respective submanifolds of $\partial N$.
\end{proposition}

\section{Algebraic definition of the twisted Blanchfield form}\label{section:algebraic-defn-TBF}

Let $\pi$ be a finitely presented group, let $R$ be a ring with involution and let $V$ be an $(R,\Z[\pi])$-bimodule.
Let
\[\xymatrix {
D \ar[r]^{j_A} \ar[d]_{j_B} & A \ar[d]^{i_A} \\ B \ar[r]_{i_B} & C
}, (\Phi, \varphi^A,\varphi^B, \chi)\]
be a 3-dimensional symmetric triad of finitely generated free left $\Z[\pi]$-modules with all maps $i_A, i_B, j_A, j_B$ split injections.  Let $S \subset R$ be a multiplicative subset for which the pair $(R,S)$ satisfies the Ore condition, so that the localisation $S^{-1}R$ is well-defined.
That is, $S$ contains no zero divisors, and for each pair $(s,r) \in S \times R$ with $s \neq 0$, there exists another pair $(s',r')$ such that $r's=s'r$.  See \cite[Chapter~II]{Stenstrom} for an introduction to Ore domains.

Define $V^* := \Hom_R(V,R)^t$, the $R$-dual, converted into a left $R$-module using the involution.
The right $\Z[\pi]$-module structure of $V^*$ is defined via $(f\cdot g)(v) = f(v\ol{g})$, where $f \in V^*$, $g \in \Z[\pi]$ and $v \in V$.
After tensoring a chain complex and its dual with $V$, boundary, coboundary and symmetric structure maps $f=\partial, \partial^*$ or $\Phi_s$ become $\Id \otimes f$, however we usually omit $\Id \otimes$ from the notation.

For future reference we record the following elementary lemma.

\begin{lemma}\label{lem:isoduals}
Let $R$ and $\mathcal{A}$ be rings with involution.
Let $V$ be an $(R,\mathcal{A})$-bimodule and let $W$ be a free finitely generated $\mathcal{A}$-module.  Define $V^* := \Hom_R(V,R)^t$ and $W^* := \Hom_{\mathcal{A}}(W,\mathcal{A})^t$.  Define the right $\mathcal{A}$-module structure on $V^*$ by $(f\cdot s)(v) = f(v\ol{s})$, where $f \in V^*$, $v \in V$ and $s \in \mathcal{A}$.  Then the map
\[ \ba{rcl} (V^*\otimes_{\mathcal{A}} W^*)&\to & (V\otimes_\mathcal{A} W)^* \\
(\phi\otimes f)&\mapsto & (v\otimes w\mapsto \phi(v\cdot f(w))\ea \]
is an isomorphism of left $R$-modules.
\end{lemma}

As an example, of an $(R,\Z[\pi])$-bimodule, suppose that $V=R^k$, and the $\Z[\pi]$-action on $V$ is defined via a representation $\a \colon \pi \to \Aut(R^k)$.
A representation $\a \colon \pi \to \Aut(R^k)$ is called \emph{unitary} if $\a(g^{-1}) = \ol{\a(g)}^T$ for each $g \in \pi$.

In more generality, let $V$ be an arbitrary $(R,\zpx)$-bimodule that possesses an $R$-sesquilinear, nonsingular inner product $\ll \cdot, \cdot \rr$, and let $\a \colon \pi \to \Aut(V)$ be a representation.  Then $\a$ is called \emph{unitary} if $\ll v\a(g),w\a(g)\rr = \ll v,w \rr$ for all $v,w \in V$ and $g \in \pi$.  (The first definition of unitary above corresponds to the standard hermitian inner product on $R^k$.) This data is equivalent to an isomorphism of $(R,\zpx)$-bimodules $\Theta \colon V \toiso V^*$.  From now on we require that $V$ is always equipped with such an isomorphism $\Theta$, or equivalently we require that the representation $\a$ be unitary.


For an $R$-module $M$ let $TM$ denote the maximal $S$-torsion submodule
\[\{m \in M \,|\, sm=0 \text{ for some } s \in S\}.\]
That this is a submodule follows from the Ore condition.
The next definition is based on~\cite[Page~185]{Ranicki:1981-1}.  See also \cite[Proposition~3.4.1]{Ranicki:1981-1} for the precise relationship (in the case $A=B=0$) between symmetric complexes and linking pairings.

\begin{definition}[Twisted Blanchfield pairing]\label{defn:chain-level-blanchfield}
The twisted Blanchfield pairing of a symmetric triad
\[\wt{\Bl} \colon TH^2(V \otimes_{\Z[\pi]} C/B) \times TH^2(V \otimes_{\Z[\pi]} C/A) \to S^{-1}R/R\]
is defined as follows.  For $[x] \in TH^2(V \otimes_{\Z[\pi]} C/A)$ and $[y] \in TH^2(V \otimes_{\Z[\pi]} C/B)$, let
\[\wt{\Bl}([y],[x]) = \smfrac{1}{s} \ol{z(\Phi_0(x))}\]
where $x \in V \otimes_{\Z[\pi]} (C/A)^2$, $y \in V \otimes_{\Z[\pi]} (C/B)^2$, $z \in V \otimes_{\Z[\pi]} (C/B)^1$, $\partial^*(z) = sy$ for some  $s \in S$.  
To evaluate $z$ on $\Phi_0(x)$ use the image of $z$ under the isomorphisms
\[V \otimes_{\Z[\pi]} (C/B)^1 \xrightarrow{\Theta \otimes \Id} V^* \otimes_{\zpx} (C/B)_1^* \xrightarrow{\text{Lemma }\ref{lem:isoduals}} (V \otimes (C/B)_1)^*\]

For a symmetric \emph{Poincar\'{e}} triad, we can also define the Blanchfield pairing on homology:
\[\Bl \colon TH_1(V \otimes_{\Z[\pi]} C/A) \times TH_1(V \otimes_{\Z[\pi]} C/B) \to S^{-1}R/R\]
via $\Bl([u],[v]) := \wt{\Bl}([\Phi_0]^{-1}([u]),[\Phi_0]^{-1}([v]))$, where $[\Phi_0]$ is the induced map on homology.
\end{definition}

Using the identification of a finitely generated free module with its double dual:
\[M \xrightarrow{\simeq} M^{**};\; x \mapsto (f \mapsto \ol{f(x)})\]
one can also write $\ol{z(\Phi_0(x))}$ as $\Phi_0(x)(z)$, if one prefers to hide the involution in the definition.

\begin{proposition}[Well-defined and linearity]\label{Prop:chainlevelBlanchfield}
The twisted Blanchfield pairing is well-defined and sesquilinear:
\[\wt{\Bl}(q[u],p[v]) = q\wt{\Bl}([u],[v]) \ol{p}.\]
\end{proposition}

Under additional assumptions the twisted Blanchfield pairing is also nonsingular and hermitian.  We will investigate these two properties in Sections~\ref{section:non-singular} and \ref{section:hermitian}.
All properties will be proven for the cohomology Blanchfield pairing $\wt{\Bl}$.  This implies the same properties for $\Bl$ in the case that we have a Poincar\'{e} triad.  Recall that a symmetric triad arising from a $3$-manifold is Poincar\'e, as in Proposition~\ref{prop:triad-from-manifold}.

\begin{proof}[Proof of Proposition~\ref{Prop:chainlevelBlanchfield}]
First we consider the dependence on the choice of chain representative.  Suppose instead we compute $\wt{\Bl}$ using $x' = x + \partial^*u$ and $y'=y+\partial^*v$, where $u \in V \otimes_{\zpx} (C/A)^1$ and $v \in V \otimes_{\zpx} (C/B)^1$.  Then $\Phi_0(x') = \Phi_0(x+\partial^*(u)) = \Phi_0(x) + \partial \Phi_0(u)$.  Also $sy' = sy + \partial^*sv = \partial^*(z+sv)$.
Then we have
\begin{align*}
 \wt{\Bl}(y',x') =&\smfrac{1}{s}\ol{(z+sv)(\Phi_0(x) + \partial\Phi_0(u))} \\
   =&  \smfrac{1}{s}\ol{z(\Phi_0(x))} + \smfrac{1}{s}\ol{(sv)(\Phi_0(x) + \partial\Phi_0(u))} + \smfrac{1}{s}\ol{z(\partial\Phi_0(u))} \\
   =&  \smfrac{1}{s}\ol{z(\Phi_0(x))} + \smfrac{1}{s}\ol{v(\Phi_0(x) + \partial\Phi_0(u))\ol{s}} + \smfrac{1}{s}\ol{(\partial^*z)(\Phi_0(u))} \\
   =&  \smfrac{1}{s}\ol{z(\Phi_0(x))} + \smfrac{1}{s}\ol{\ol{s}}\ol{v(\Phi_0(x) + \partial\Phi_0(u))} + \smfrac{1}{s}\ol{(sy)(\Phi_0(u))} \\
  =&  \smfrac{1}{s}\ol{z(\Phi_0(x))} + \smfrac{s}{s}\ol{v(\Phi_0(x) + \partial\Phi_0(u))} + \smfrac{s}{s}\ol{y(\Phi_0(u))} \\
  =&  \smfrac{1}{s}\ol{z(\Phi_0(x))} + \ol{v(\Phi_0(x) + \partial\Phi_0(u))} + \ol{y(\Phi_0(u))} \\
  =&  \smfrac{1}{s}\ol{z(\Phi_0(x))} = \wt{\Bl}(y,x)
\end{align*}
where the penultimate equality follows since all but the first terms lie in $R$, and the Blanchfield pairing is valued in $S^{-1}R/R$.

To show that the definition is independent of the choice of $s$ and $z$, suppose that there are also $s' \in S, z' \in V \otimes_{\Z[\pi]} (C/B)^1$ such that $\partial^*(z') = s'y$.
Since $\Phi_0(x)$ is a torsion element of $H_1(V \otimes_{\Z[\pi]} C/B)$, there is a chain $w \in V \otimes_{\Z[\pi]} (C/B)_2$ and an $r \in S$ such that $\partial(w) = r\Phi_0(x)$.  Then:
\begin{eqnarray*} \smfrac{1}{s}\ol{z(\Phi_0(x))} - \smfrac{1}{s'}\ol{z'(\Phi_0(x))}
& = &  \left(\smfrac{1}{s} \ol{z(\Phi_0(x))} - \smfrac{1}{s'} \ol{z'(\Phi_0(x))}\right)\smfrac{\ol{r}}{\ol{r}} \\
 =  \left(\smfrac{1}{s} \ol{z(\Phi_0(x))}\,\ol{r} - \smfrac{1}{s'} \ol{z'(\Phi_0(x))}\,\ol{r}\right)\smfrac{1}{\ol{r}}
 &= & \left(\smfrac{1}{s} \ol{r(z(\Phi_0(x)))} - \smfrac{1}{s'} \ol{r(z'(\Phi_0(x)))}\right)\smfrac{1}{\ol{r}}\\
 =  \left(\smfrac{1}{s} \ol{z(r\Phi_0(x))} - \smfrac{1}{s'} \ol{z'(r\Phi_0(x))}\right)\smfrac{1}{\ol{r}}
& = & \left(\smfrac{1}{s} \ol{z(\partial w)} - \smfrac{1}{s'} \ol{z'(\partial w)}\right)\smfrac{1}{\ol{r}}\\
 =  \left(\smfrac{1}{s} \ol{\partial^*(z)(w)} - \smfrac{1}{s'} \ol{\partial^*(z')(w)}\right)\smfrac{1}{\ol{r}}
& = & \left(\smfrac{1}{s} \ol{(sy)(w)} - \smfrac{1}{s'} \ol{(s'y)(w)}\right)\smfrac{1}{\ol{r}}\\
 =  \left(\smfrac{1}{s} \ol{(y)(w)\ol{s}} - \smfrac{1}{s'} \ol{(y)(w)\ol{s'}}\right)\smfrac{1}{\ol{r}}
& = & \left(\smfrac{1}{s} \ol{\ol{s}}\ol{(y)(w)} - \smfrac{1}{s'} \ol{\ol{s'}}\ol{(y)(w)}\right)\smfrac{1}{\ol{r}}\\
 =  \left(\ol{y(w)} - \ol{y(w)}\right) \smfrac{1}{\ol{r}} &=& 0.\end{eqnarray*}
Furthermore, for $p \in R$:
\[\wt{\Bl}(y,px) = \smfrac{1}{s}\ol{(z)(\Phi_0(px))} = \smfrac{1}{s} \ol{p z(\Phi_0(x))} = \smfrac{1}{s} \ol{z(\Phi_0(x))} \ol{p} = \wt{\Bl}(y,x)\ol{p},\]
so that $\wt{\Bl}$ is conjugate-linear in the second variable.
Then for $q \in R$, suppose as usual that $\partial^*(z) = sy$.  By the Ore condition, there exist $q' \in R$ and $s' \in S$ such that $q's=s'q$.  Then
$\partial^*(q'z) = q'sy = s'(qy)$.  Note that $q's = s'q$ implies that $\smfrac{1}{s'}q' = q \smfrac{1}{s}$.  Therefore
\begin{eqnarray*}
  \wt{\Bl}(qy,x) &=& \smfrac{1}{s'} \ol{(q'z)(\Phi_0(x))} =  \smfrac{1}{s'} \ol{z(\Phi_0(x)) \ol{q'}} = \smfrac{1}{s'} q' \ol{z(\Phi_0(x))} \\ &=& q\smfrac{1}{s} \ol{z(\Phi_0(x))} = q\wt{\Bl}(y,x).
\end{eqnarray*}
Linearity under addition is a straightforward verification and is left to the reader.
\end{proof}

This completes the proof of Theorem~\ref{theorem:main}~(\ref{item:mainthm-sesqui}) from the introduction.  It also proves Theorem~\ref{theorem:main}~(\ref{item:mainthm-defined}), except for the claim that the pairing of Definition~\ref{defn:chain-level-blanchfield} coincides with the homological definition of the pairing given in the introduction.

\subsection{Nonsingular Blanchfield pairings}\label{section:non-singular}

The twisted Blanchfield pairing of a symmetric Poincar\'{e} triad is often nonsingular.
The next proposition proves Theorem~\ref{theorem:main}~(\ref{item:mainthm-nonsingular}).

\begin{proposition}[Non-singularity]\label{prop:non-singular}
Suppose that \[\Ext^1_R(H_0(V \otimes_{\zpx}C/B),S^{-1}R/R) \cong 0.\]
Then the twisted Blanchfield pairing of a symmetric Poincar\'{e} triad is nonsingular; that is, the adjoint map $$\wt{\Bl} \colon TH^2(V \otimes_{\Z[\pi]} C/A) \to \Hom_R(TH^2(V \otimes_{\Z[\pi]} C/B),S^{-1}R/R)^t$$ is an isomorphism.
\end{proposition}

Recall that for a left $R$-module $P$ and an $(R,R')$-bimodule $Q$, as in Definition~\ref{defn:dual-complex}, we have that $\Hom_{R}(P,Q)$ is naturally a right $R'$-module, but we convert it to a left $R'$-module $\Hom_{R}(P,Q)^t$  using the involution of $R'$.  

As remarked in the introduction, if $R$ is a principal ideal domain and $S$ consists of the nonzero elements of $R$, then since $S^{-1}R/R$ is divisible it is an injective $R$-module (see \cite[I.6.10]{Stenstrom}).  Therefore $\Ext_R^i(M,S^{-1}R/R) = 0$ for any $i>0$, as can be seen by computing $\Ext$ using the injective resolution of length zero of the second variable.  Thus the pairing of a Poincar\'{e} triad is nonsingular.

\begin{proof}[Proof of Proposition~\ref{prop:non-singular}]
For an $(R,R)$-bimodule $M$, for the duration of this proof we write \[H^i(C/B;M) := H_i(M \otimes_R (V \otimes_{\zpx} C/B)^*).\]
The Bockstein homomorphism $\beta$ associated to the short exact sequence
\[0 \to R \to S^{-1}R \to S^{-1}R/R \to 0\]
yields a long exact sequence
\[H^1(C/B;S^{-1}R) \xrightarrow{\gamma} H^1(C/B;S^{-1}R/R) \xrightarrow{\beta} H^2(C/B;R) \to H^2(C/B;S^{-1}R).\]
Since $[y] \in H^2(C/B;R)$ is $S$-torsion, $[y] \mapsto 0 \in H^2(C/B;S^{-1}R)$, so there is a $[z] \in H^1(C/B;S^{-1}R/R)$ such that $\beta(z) =y$.  The chain $z$ is determined up to addition of an element $\gamma(w)$ where $w \in H^1(C/B;S^{-1}R)$.    The $\Ext^1$ vanishing hypothesis implies, by universal coefficients, that we have an isomorphism
\[H^1(C/B;S^{-1}R/R) \toiso \Hom_R(H_1(C/B;R),S^{-1}R/R)^t.\]
Then Poincar\'{e} duality induces an isomorphism \[\Hom_R(H_1(C/B;R);S^{-1}R/R)^t \toiso \Hom_R(H^2(C/A;R),S^{-1}R/R)^t.\]
Denote the image of $z + \gamma(w)$ in $\Hom_R(H^2(C/A;R),S^{-1}R/R)^t$ by $z^* + \gamma(w)^*$.  For $x \in TH^2(C/A;R)$, we have that $\gamma(w)^*(x) =0$ since $\gamma(w)^*$ originates from $\Hom_R(H_1(C/B;R),S^{-1}R)^t$, and homomorphisms to $S^{-1}R$ vanish on $S$-torsion, since $S$ contains no zero divisors.
Therefore the indeterminacy in $z$ has no effect, and the combination of the Bockstein, universal coefficients and Poincar\'{e} duality induce a bijection:
\[\Gamma \colon TH^2(C/B;R) \to \Hom_R(TH^2(C/A;R),S^{-1}R/R)^t.\]
It is straightforward to check that $\Gamma$ is a module homomorphism, and therefore is an isomorphism.
Inspection of the sequence of isomorphisms, together with the fact that the symmetric structure map $\Phi_0$ is the chain level map inducing the Poincar\'e duality isomorphism on homology, yields that the isomorphism $\Gamma$ is adjoint to~$\wt{\Bl}$.  This completes the proof of Proposition~\ref{prop:non-singular}.
\end{proof}

The sequence of isomorphisms used in the above proof leads to the definition of the twisted Blanchfield pairing given in the introduction, by applying Poincar\'{e} duality at the beginning and at the end, as in Definition~\ref{defn:chain-level-blanchfield}.
This completes the proof of Theorem~\ref{theorem:main}~(\ref{item:mainthm-defined}) from the introduction.

\subsection{Hermitian triads and hermitian Blanchfield pairings}\label{section:hermitian}

\begin{definition}[Hermitian triad]\label{defn:hermitian-triad}
Suppose that a symmetric triad is chain equivalent to a triad (the maps of the triad are still required to be split injections) for which there is a chain equivalence $\sigma\colon C/B \xrightarrow{\simeq} C/A$, such that the following diagram commutes up to homotopy:
\[\xymatrix @R-0.5cm {& C \ar[dl]_-{q_B} \ar[dr]^-{q_A} & \\  C/B \ar[rr]^-{\simeq}_-{\sigma} && C/A.
  }\]
Then we say that the symmetric triad is \emph{hermitian}.
\end{definition}

It is worth noting that in order to check this criterion in practice, for a $3$-manifold, one has to take extra care with basepoints and basing paths, due to their appearance in the preamble to Proposition~\ref{prop:triad-from-manifold}.  See the proof of Proposition~\ref{prop:boundary-split} for an example of this principle in practice.

\begin{lemma}\label{lemma:hermitian-chain-homotopy}
For a hermitian triad, the homomorphisms
\[(C/A)^{3-r} \xrightarrow{q_A^*} C^{3-r} \xrightarrow{\Phi_0} C_{r} \xrightarrow{q_B} (C/B)_r \xrightarrow{\sigma} (C/A)_r\]
and
\[(C/A)^{3-r} \xrightarrow{\sigma^*} (C/B)^{3-r} \xrightarrow{q_B^*} C^{3-r} \xrightarrow{\Phi^*_0} C_{r} \xrightarrow{q_A} (C/A)_r.\]
determine chain homotopic chain maps $(C/A)^{3-*} \to (C/A)_*$.
\end{lemma}

We remark that the maps $\Phi_0, \Phi_0^*$ are not chain maps by themselves, only the composition $q_B \Phi_0 q_A^*$ and its dual $q_A \Phi_0^* q_B^*$ are chain maps, as we will show in the proof of the lemma.
The proof that the two maps are chain homotopic uses a combination of the chain homotopy witnessing the homotopy commutativity of the diagram from Definition~\ref{defn:hermitian-triad}, and the higher symmetric structure maps $\Phi_1 \colon C^{3-r} \to C_{r+1}$.

\begin{proof}
We begin by showing that the first composition is a chain map.
  In this proof let $\p\Phi := \varphi^A \cup_{\chi} \varphi^B$.  The symmetric pair equations tell us that
  \[\p_C\Phi_0 + (-1)^r\Phi_0\delta_C = e \p\Phi_0 e^* \colon C^{2-r} \to C_{r}.\]
 Here \[e = \begin{pmatrix}
i_A & 0 & -i_B
\end{pmatrix} \colon A_{r} \oplus D_{r-1} \oplus B_r \to C_r.\]  The maps $\sigma$, $q_A$ and $q_B$ are chain maps so we obtain:
\[\sigma q_B(\p_C\Phi_0 + (-1)^r\Phi_0\delta_C)q_A^* = \p_C \sigma q_B\Phi_0 q_A^* + (-1)^r\sigma q_B\Phi_0 q_A^*\delta_C. \]
In order to show that $\sigma q_B\Phi_0 q_A^*$ is a chain map, it therefore suffices to show that $\sigma q_B e (\p\Phi)_0 e^* q_A^*$ vanishes.

\begin{claim*}
We have $q_B e (\p\Phi)_0 e^* q_A^* =0$.
\end{claim*}

Since $i_A \colon D \to A$ is a split injection, the image of $e^* q_A^* \colon (C/A)^{3-r} \to A^{3-r} \oplus D^{2-r} \oplus B^{3-r}$ is contained in $B^{3-r}$.  The symmetric structure $(\p\Phi)_0$ is given by
\[\bp \varphi^A_0 & 0 & 0 \\
(-1)^{r}\chi_0 i_A^* & 0 & 0 \\
0 & i_B \chi_0 & \varphi^B_0
\ep. \]
This maps $(0,0,b)^T \in A^{3-r} \oplus D^{2-r} \oplus B^{3-r}$ to \[(0,0,\varphi^B_0(b))^T \in A_{r-2} \oplus D_{r-2} \oplus B_{r-1},\] which is in the kernel of $C_{r-1} \to (C/B)_{r-1}$.  This completes the proof of the claim that $q_B e (\p\Phi)_0 e^* q_A^* =0$.

Thus $\sigma q_B\Phi_0 q_A^*$ is a chain map $(C/A)^{3-*} \to (C/A)_*$.  It follows that its dual is also such a chain map.

Now we need to show that the two compositions in the statement of the lemma are chain homotopic chain maps.  To begin, we recall the symmetric pair structure equation
\begin{equation}\label{equation:symm-pair-Phi-1}
  \Phi_0 - \Phi_0^*= (-1)^r\Phi_1\delta_C - \p_C\Phi_1 +e(\p\Phi)_1e^* \colon C^{3-r} \to C_r.
\end{equation}
Now $(\p\Phi)_1$ is given by
\[\bp \varphi^A_1 & 0 & 0 \\
(-1)^{r}\chi_1 i_A^* & (-1)^r\chi_0^* & 0 \\
0 & -i_B \chi_1 & \varphi^B_1
\ep. \]
Compared to $(\p\Phi)_0$, there is an extra nonzero term in the $(2,2)$ position.  Nevertheless, an analogous argument to that used for the previous claim can be applied to show that $q_B e (\p\Phi)_1 e^* q_A^* =0$.
Now we have:
\begin{align*}
  &\sigma q_B \Phi_0 q_A^* - q_A \Phi_0^* q_B^* \sigma^* \\
  =& \sigma q_B \big(\Phi_0^* + (-1)^r\Phi_1\delta_C - \p_C\Phi_1 +e(\p\Phi)_1e^*\big) q_A^* - q_A \Phi_0^* q_B^* \sigma^*\\
 =& \sigma q_B \Phi_0^*q_A^* + (-1)^r\sigma q_B\Phi_1\delta_C q_A^* - \sigma q_B\p_C\Phi_1 q_A^* +\sigma q_Be(\p\Phi)_1e^* q_A^* - q_A \Phi_0^* q_B^* \sigma^*\\
 =& \sigma q_B \Phi_0^*q_A^* + (-1)^r\sigma q_B\Phi_1 q_A^* \delta_{C/A} - \p_{C/A}\sigma q_B\Phi_1 q_A^*  - q_A \Phi_0^* q_B^* \sigma^*.
\end{align*}
The first equality follows from (\ref{equation:symm-pair-Phi-1}).  The second equality is just expanding the brackets.  The third equality uses that $q_B e (\p\Phi)_1 e^* q_A^* =0$ and that $q_A$, $q_B$ and $\sigma$ are chain maps, and so commute with boundary maps.  Now use that there is a chain homotopy $k \colon C_r \to (C/A)_{r+1}$ for which
\[\sigma q_B -q_A = k \p_{C} + \p_{C/A} k \colon C_r \to (C/A)_{r}.\]
Note that by dualising we also have \[q_B^* \sigma^* -q_A^* = (-1)^r \delta_{C}k^* + (-1)^{r+1} k^* \delta_{C/A} \colon (C/A)^{r} \to C^r.\] In addition, define $L:= \sigma q_B\Phi_1 q_A^*$.  Thus:
\begin{align*}
 & \sigma q_B \Phi_0^*q_A^* + (-1)^r\sigma q_B\Phi_1 q_A^* \delta_{C/A} - \p_{C/A}\sigma q_B\Phi_1 q_A^*  - q_A \Phi_0^* q_B^* \sigma^*\\
 =& \big(q_A + k \p_{C} + \p_{C/A} k\big)\Phi_0^*q_A^*  + (-1)^r L \delta_{C/A} - \p_{C/A} L \\  & - q_A \Phi_0^* \big(q_A^* + (-1)^{r}\delta_{C}k^* + (-1)^{r+1}k^* \delta_{C/A}\big)\\
 =& q_A\Phi_0^*q_A^* + k \p_{C}\Phi_0^*q_A^* + \p_{C/A} k \Phi_0^*q_A^*  + (-1)^r L \delta_{C/A} \\
  & - \p_{C/A} L  - q_A \Phi_0^* q_A^* - (-1)^{r} q_A \Phi_0^*\delta_{C}k^* - (-1)^{r+1} q_A \Phi_0^*k^* \delta_{C/A}\\
 =& k \p_{C}\Phi_0^*q_A^* + \p_{C/A} k \Phi_0^*q_A^*  + (-1)^r L \delta_{C/A} - \p_{C/A} L  \\ & + (-1)^{r+1} q_A \Phi_0^*\delta_{C}k^* + (-1)^{r} q_A \Phi_0^*k^* \delta_{C/A}\\
 =& (-1)^{r+1} k \Phi_0^*q_A^*\delta_{C/A} + \p_{C/A} k \Phi_0^*q_A^*  + (-1)^r L \delta_{C/A} - \p_{C/A} L \\ & - \p_{C/A}q_A \Phi_0^*k^* + (-1)^{r} q_A \Phi_0^*k^* \delta_{C/A}.
 \end{align*}
In the first equality we substitute $L$ for $\sigma q_B\Phi_1 q_A^*$ and we substitute $\sigma q_B = q_A+ k \p_{C} + \p_{C/A} k$ and $q_B^* \sigma^*= q_A^* + \delta_{C}k^* + k^* \delta_{C/A}$.  In the second equality we expand the brackets.  In the third equality we delete cancelling copies of $q_A \Phi_0^* q_A^*$.  In the fourth equality we use that $q_A$ and $\Phi_0^*$ are chain maps, to commute them with certain boundary and coboundary maps.
Now define $J = k \Phi_0^*q_A^*$ and $K = q_A \Phi_0^*k^*$.
Then making these substitutions and rearranging, we obtain:
\begin{align*}
& (-1)^{r+1} k \Phi_0^*q_A^*\delta_{C/A} + \p_{C/A} k \Phi_0^*q_A^*  + (-1)^r L \delta_{C/A} - \p_{C/A} L \\ & - \p_{C/A}q_A \Phi_0^*k^*
+ (-1)^{r} q_A \Phi_0^*k^* \delta_{C/A}\\
=&  (-1)^{r+1} J \delta_{C/A} + \p_{C/A} J  + (-1)^r L \delta_{C/A} - \p_{C/A} L - \p_{C/A} K +(-1)^{r} K \delta_{C/A}\\
=&   \p_{C/A}(J  - K - L) + (-1)^{r+1}( J  - K - L) \delta_{C/A}.
 \end{align*}
Thus $H=J-K+L$ is the chain homotopy we seek, with
\[\sigma q_B \Phi_0 q_A^* - q_A \Phi_0^* q_B^* \sigma^* = \p_{C/A}H + (-1)^{r+1}H \delta_{C/A}.\]
\end{proof}

Using $\sigma$, we can identify $TH^2(V \otimes_{\Z[\pi]} C/B)$ and $TH^2(V \otimes_{\Z[\pi]} C/A)$, and whence define a twisted Blanchfield pairing
\[\wt{\Bl}\colon TH^2(V \otimes_{\Z[\pi]} C/A) \times TH^2(V \otimes_{\Z[\pi]} C/A) \to S^{-1}R/R.\]
In the proof of Proposition~\ref{prop:Hermitian} below, we abuse but simplify notation by referring to the chain maps and the chain homotopy given to us by Lemma~\ref{lemma:hermitian-chain-homotopy} as $\Phi_1 \colon \Phi_0 \sim \Phi_0^*\colon (C/A)^{3-*} \to (C/A)_*$.

\begin{proposition}[Hermitian]\label{prop:Hermitian}
The twisted Blanchfield pairing on $TH^2(V \otimes_{\Z[\pi]} C/A)$ associated to a hermitian triad is a hermitian pairing: $\wt{\Bl}(x,y) = \ol{\wt{\Bl}(y,x)}$ for all $x,y \in TH^2(C/A;V)$.
\end{proposition}

\noindent Note that we can define a sesquilinear hermitian pairing without the hypothesis that the symmetric triad be Poincar\'{e}.
The next proposition proves Theorem~\ref{theorem:main}~(\ref{item:mainthm-hermitian}).

\begin{proof}[Proof of Proposition~\ref{prop:Hermitian}]
First, we claim that we can calculate $\wt{\Bl}$ using $\Phi_0^* \colon (C/A)^2 \to (C/A)_1$ instead of $\Phi_0$.  This follows from the following computation:
\begin{align*}\smfrac{1}{s} \ol{z (\Phi_0(x))} - \smfrac{1}{s} \ol{z (\Phi_0^*(x))} &= \smfrac{1}{s} \ol{z (\Phi_0 - \Phi_0^*)(x)}
 =  \smfrac{1}{s} \ol{z((\partial \Phi_1 - \Phi_1 \partial^*)(x))} \\
 =  \smfrac{1}{s} \ol{z(\partial \Phi_1(x))}
&=  \smfrac{1}{s} \ol{(\partial^*(z))(\Phi_1(x))}
 =  \smfrac{1}{s} \ol{(s y)(\Phi_1(x))}\\
 = \smfrac{1}{s} \ol{y(\Phi_1(x))\ol{s}}
 &=  \smfrac{1}{s} \ol{\ol{s}}\ol{y(\Phi_1(x))}
 =  \ol{y(\Phi_1(x))} = 0 \in S^{-1}R/R
\end{align*}
since $\ol{y(\Phi_1(x))} \in R$.  Now, suppose we also have an $r \in S, w \in V \otimes_{\zpx} (C/A)^1$, such that $\partial w = rx$.  Then:
\begin{align*} &\wt{\Bl}(y,x)  = \smfrac{1}{s}\,\ol{z(\Phi_0^*(x))}
 =  \smfrac{1}{s}\,\ol{z(\Phi_0^*(x))}\,\ol{r}\, \smfrac{1}{\ol{r}}
 =  \smfrac{1}{s}\,\ol{r z(\Phi_0^*(x))}\, \smfrac{1}{\ol{r}}
 =  \smfrac{1}{s}\,\ol{z(\Phi_0^*(rx))}\, \smfrac{1}{\ol{r}} \\
 =&   \smfrac{1}{s}\,\ol{z(\Phi_0^*(\partial^* w))}\, \smfrac{1}{\ol{r}}
 =  \smfrac{1}{s}\,\ol{z(\partial\Phi_0^*(w))}\, \smfrac{1}{\ol{r}}
 =  \smfrac{1}{s}\,\ol{\partial^*z(\Phi_0^*(w))}\, \smfrac{1}{\ol{r}}
 =  \smfrac{1}{s}\,\ol{(s y)(\Phi_0^*(w))}\, \smfrac{1}{\ol{r}} \\
 =&  \smfrac{1}{s}\,\ol{(y)(\Phi_0^*(w))\ol{s}}\, \smfrac{1}{\ol{r}}
 =  \smfrac{1}{s}\,\ol{\ol{s}}\,\ol{y(\Phi_0^*(w))}\, \smfrac{1}{\ol{r}}
 =  \ol{y(\Phi_0^*(w))}\, \smfrac{1}{\ol{r}}
 =  \Phi_0^*(w)(y)\, \smfrac{1}{\ol{r}}
 =  w(\Phi_0(y))\, \smfrac{1}{\ol{r}} \\
 =&  \ol{\smfrac{1}{r}\,\ol{w(\Phi_0(y))}} = \ol{\wt{\Bl}(x,y)},
\end{align*}
which shows that~$\wt{\Bl}$ is hermitian.  This completes the proof of Proposition~\ref{prop:Hermitian}.
\end{proof}

\begin{remark}
  The fact that one may use $\Phi_0^*$ to compute the twisted Blanchfield form may sometimes be exploited in computations, when the map $\Phi_0^*\colon (C/A)^2 \to (C/A)_1$ may be simpler than $\Phi_0 \colon (C/A)^2 \to (C/A)_1$.
\end{remark}

\begin{corollary}
Let $M$ be a closed oriented $3$-manifold.  For any unitary representation over an Ore domain $R$, the associated twisted Blanchfield pairing of $M$ is hermitian.
\end{corollary}

\begin{proof}
  In the case of a closed $3$-manifold, the chain complexes $A$, $B$ and $D$ of the associated symmetric triad all vanish.
\end{proof}

The corollary, together with Proposition~\ref{prop:non-singular}, gives Proposition~\ref{propn-closed-3-mfld-intro} from the introduction.

\section{Recovering the classical Blanchfield pairing}\label{section:recovering-classical-blanchfield}

Let $K$ be an oriented knot in $S^3$ and let $X_K = S^3 \sm \nu K$ be the knot exterior.
To recover the classical Blanchfield pairing on $X_K$, define $\pi := \pi_1(X_K)$, use $R = \Z[t,t^{-1}]$, let $S = \Z[t,t^{-1}]\sm \{0\}$, let $V=R=\Z[t,t^{-1}]$, and let $\a \colon \pi \to \Z = \langle t \rangle$ be the abelianisation.  Then $\Z[\pi]$ acts on $V$ via $\Z[\pi] \to \Z[t,t^{-1}]$ and then right multiplication. Write $\L := \Z[t,t^{-1}]$ and $Q := \Q(t)$.  Note that $H_1(X_K;\L) \cong H^2(X_K,\partial X_K;\L)$ is a $\L$-torsion module~\cite[Corollary~1.3]{Levine-knot-modules}.

First consider the symmetric hermitian triad
\[\xymatrix {
0 \ar[r] \ar[d] & 0 \ar[d] \\ 0 \ar[r] & C_*(X_K,\partial X_K;\L)
}, (\Delta([X_K,\partial X_K])/\Delta([\partial X_K]),0,0,0).\]
From this we obtain a sesquilinear hermitian Blanchfield pairing
\[\Bl_1 \colon H^2(X_K,\partial X_K;\L) \times H^2(X_K,\partial X_K;\L) \to Q/\L.\]
Next consider the symmetric Poincar\'{e} triad
\[\xymatrix {
0 \ar[r] \ar[d] & C_*(\partial X;\L) \ar[d] \\ 0 \ar[r] & C_*(X_K;\L)
}, (\Delta([X_K,\partial X_K],\Delta([\partial X_K]),0,0).\]
From this we obtain a sesquilinear, nonsingular Blanchfield pairing
\[\Bl_2 \colon H^2(X_K;\L) \times H^2(X_K,\partial X_K;\L) \to Q/\L.\]
The $\Ext^1$ condition for nonsingularity, from Proposition~\ref{prop:non-singular}, is easily seen to be satisfied since $H_0(X_K,\partial X_K;\L)$ vanishes\footnote{That the zeroth homology vanished was not crucial.  For pairings over $\L$ associated to closed manifolds, the zeroth homology is $\Z$, and we have a Bockstein exact sequence $0=\Ext^1_{\L}(\Z,Q) \to  \Ext^1_{\L}(\Z,Q/\L) \to \Ext^2_{\L}(\Z,\L) =0$, from which it follows that the central module is trivial, so the $\Ext^1$ condition for nonsingularity is satisfied.}.  Note that $C_*(X_K;\L)$ and $C_*(X_K,\partial X;\L)$ are not chain homotopy equivalent -- they have different zeroth homology groups, for example -- so we cannot directly apply Proposition \ref{prop:Hermitian} to deduce that $\Bl_2$ is hermitian.
In the long exact sequence of the pair $(X_K,\partial X_K)$ we have zero maps as shown
\[H^1(\partial X_K;\L) \xrightarrow{0} H^2(X_K,\partial X_K;\L) \toiso H^2(X_K;\L) \xrightarrow{0} H^2(\partial X_K;\L),\]
so the middle map is an isomorphism.
The adjoints of the two Blanchfield pairings above fit into the following diagram.
\[\xymatrix @C+1.4cm{H^2(X_K,\partial X_K;\L) \ar[d]_{\cong} \ar[dr]_-{\Bl_2}^-{\cong} & \\
\Hom_{\L}(H_1(X_K,\partial X_K;\L),Q/\L) \ar[d]_-{(\Phi^1_0)^*} \ar[r]^-{\cong}_-{(\Phi^2_0)^*} & \Hom_{\L}(H^2(X_K;\L),Q/\L) \ar[dl]^{\cong}  \\
\Hom_{\L}(H^2(X_K,\partial X_K;\L),Q/\L)  &  } \]
Some explanation of the diagram follows.
\begin{enumerate}
\item The top left vertical map is given by the combination of the Bockstein and universal coefficients described in the proof of Proposition~\ref{prop:non-singular}.
\item The bottom right diagonal map arises from the long exact sequence of the pair $(X_K,\partial X_K)$, as above.
\item The maps $(\Phi^1_0)^*$ and $(\Phi^2_0)^*$ are induced by the symmetric structure maps of the triads above that gave rise to $\Bl_1$ and $\Bl_2$ respectively.
\item  The bottom triangle commutes since, by definition, $\Phi^1_0$ and $\Phi^2_0$ factor through the maps in the sequence of a pair; see Section~\ref{section:symm-Pe-triads}.
\item The top triangle commutes by definition of $\Bl_2$.
\item The adjoint of the first pairing $\Bl_1$ is given by the composition of the two vertical maps.
\end{enumerate}
It follows immediately from the diagram and the above observations that  $\Bl_1$ is nonsingular.
Applying Poincar\'{e} duality we obtain the classical Blanchfield pairing
\[\Bl \colon H_1(X_K;\L) \times H_1(X_K;\L) \to Q/\L.\]
We have verified that it is hermitian and nonsingular (see \cite{Friedl-Powell-2015-1} for an alternative argument).
Inspection of the essential properties that we have used in the argument yields the following potentially useful proposition, which was Proposition~\ref{propn-3-mfld-bdy-intro} in the introduction.

\begin{proposition}
  Let $X$ be an oriented $3$-manifold with boundary, let $(R,S)$ be an Ore pair, let $Q= S^{-1}R$ and let $V$ be an $(R,\Z[\pi])$-bimodule with an inner product and a unitary representation $\a \colon \pi \to \Aut(V)$, such that at least one of $\Ext^1_R(H_0(X,\partial X; V),S^{-1}R/R)$ or
  $\Ext^1_R(H_0(X; V),S^{-1}R/R)$ vanishes.
  Moreover, suppose that the quotient map $C_*(X;V) \to C_*(X,\partial X;V)$ induces an isomorphism $TH^2(X,\partial X;V) \toiso TH^2(X;V)$.  Then the Blanchfield pairing $\Bl \colon TH_1(X;V) \times TH_1(X;V) \to Q/R$ is hermitian and nonsingular.
\end{proposition}

\begin{remark}
For example, as in \cite{Hillman:2012-1-second-ed}, take $X$ to be the exterior of an $m$-component oriented link $L$ in $S^3$, let $R$ be the ring obtained from $\Z[t_1^{\pm 1}, \dots, t_m^{\pm 1}]$ by inverting $(t_i-1)$ for $i=1,\dots,m$, and let $V=R$, with the action induced by the abelianisation $\Z[\pi_1(X)] \to \Z[\Z^m]=\Z[t_1^{\pm 1}, \dots, t_m^{\pm 1}] \to R$.  Then since $C_*(\partial X;R)$ is contractible, the proposition applies and the corresponding Blanchfield form of the link $L$ is hermitian and nonsingular.
\end{remark}

We can also recover the torsion linking pairing on a $k$-fold branched cover of a knot from the knot exterior.  We obtain the linking pairing on a finite branched covering space by using a representation whose dimension is equal to the degree of the covering space, in terms of data in the base space, instead of first understanding the covering space geometrically and then computing its linking pairing.

Compose the abelianisation homomorphism with the quotient map $\Z \to \Z_k$, for some natural number $k>1$, to obtain a homomorphism $\pi_1(X_K) \to \Z_k$.  Let $R = \Z$, let $S= \Z \sm \{0\}$, let $V=\Z^k$ and let $\a \colon \pi_1(X_K) \to \Z_k \to \Aut(V)=\GL(k,\Z)$ be determined by the regular representation.  Now, with two meridians of a knot dividing the boundary of $X_K$ in two, we have $\partial X_K = S^1 \times D^1 \cup_{S^1 \times S^0} S^1 \times D^1$, so we obtain:
\[\Bl_k \colon H_1(X_K,S^1\times D^1;\Z^k) \times H_1(X_K,S^1\times D^1;\Z^k) \to \Q/\Z.\]
Note that $\partial X_K \sm S^1 \times D^1 \cong S^1 \times D^1$, and moreover the longitude of $K$ maps trivially to $\Z_k$.  Thus the two boundary components $A=C_*(S^1\times D^1;\Z^k) = B$ are isomorphic, and there is a self -diffeomorphism of $X_K$, that is isotopic to the identity map, taking that switches the two $S^1 \times D^1$ components of $\partial X_K$.  It follows that we have a hermitian triad.
Now, $H_1(X_K,S^1;\Z^k)$ is isomorphic to $H_1(\Sigma_k;\Z)$, where $\Sigma_k$ is the $k$-fold cover of $S^3$ branched over $K$.  We obtain the classical linking form on this 3-manifold.  We also note that $H_1(X_K,S^1;\Z^k)$ is $\Z$-torsion, since $\Sigma_k$ is a rational homology sphere.

\section{Toroidal boundary divided into two annuli}\label{section:toroidal-boundary-divided}

In this section we consider the following situation.  Let $N$ be a 3-manifold whose boundary is a union of tori $\partial N = \bigsqcup^m S^1 \times S^1$, and let $A = B = \bigsqcup^m S^1 \times D^1 \subset \partial N$ be such that each torus boundary component of $N$ is decomposed as $S^1 \times D^1 \cup_{S^1 \times S^0} S^1 \times D^1$, with one annulus lying in~$A$ and one lying in~$B$.  As usual we denote $\pi=\pi_1(N)$.
With this setup, we can also obtain a hermitian pairing for certain representations.  It may be that $TH^2(N,\partial N;V)$ and $TH^2(N;V)$ are not isomorphic, so Proposition~\ref{propn-3-mfld-bdy-intro} cannot be applied, and therefore one is led to decompose the boundary tori as above.

The key point is that the two inclusions $A \to \partial N$ and $B \to \partial N$ are homotopic maps, and in fact the map $B \to N$ can be obtained from the inclusion $A \to N$ composed with a self-diffeomorphism of $N$ that is isotopic to the identity.  Simply rotate each boundary torus by $\pi$ radians, in such a way as to interchange $A$ and $B$, and then interpolate between this diffeomorphism and the identity in a collar neighbourhood of $\partial N$.  More precisely, for each torus $S^1 \times S^1 \subset \partial N$, given coordinates on a collar neighbourhood $S^1 \times S^1 \times I$, $(\theta,\phi,t) \mapsto (e^{i\theta},e^{i\phi},t)$, for $\theta,\phi \in [0,2\pi)$ and $t \in [0,1]$, where $S^1 \times S^1 \times \{1\} \subset \partial N$, we have a diffeomorphism
\[\ba{rcl}
\mathcal{S} \colon S^1 \times S^1 \times I & \to & S^1 \times S^1 \times I \\
(\theta,\phi,t) &\mapsto & (\theta,\phi+ t\pi, t).
\ea\]
On $S^1 \times S^1 \times \{0\}$, $\mathcal{S}$ is the identity, while on $S^1 \times S^1 \times \{1\}$ it is the promised rotation, isotopic to the identity, that interchanges $A$ and $B$.  The isotopy to the identity is realised as $t$ varies from $0$ to $1$.

Choose a point $p \in S^1$ and choose a parametrisation of each connected component of $A$ and $B$ as $S^1 \times D^1_{A}$, respectively $S^1 \times D^1_B$, such that $(p,0_A) \sim (p,1_B)$ and $(p,0_B) \sim (p,1_A)$ in the gluing $S^1 \times D^1_A \cup_{S^1 \times S^0} S^1 \times D^1_B$.
In the following, we refer to the curve on $S^1 \times S^1 \subset \partial N$ defined as the union $\{p\} \times D^1_A \cup \{p\} \times D^1_B$ as a \emph{longitude} of the boundary component.

\begin{proposition}\label{prop:boundary-split}
Let $N$ be an oriented 3-manifold with toroidal boundary, and let $A$ and $B$ be as above.  Let $V$ be an $(R,\Z[\pi])$-bimodule with a unitary representation $\a \colon \pi \to \Aut(V)$, such that $\Ext^1_R(H_0(N,A; V),S^{-1}R/R) =0$ (equivalently $\Ext^1_R(H_0(N,B; V),S^{-1}R/R) =0$).
Then using the slide diffeomorphism $\mathcal{S}$ defined above (which depends on the parametrisation) to induce an identification
\[TH_1(N,A;V) \toiso TH_1(N,B;V),\] we obtain a twisted Blanchfield pairing \[\Bl  \colon TH_1(N,A;V) \times TH_1(N,A;V) \to S^{-1}R/R\]
that is hermitian and nonsingular.
\end{proposition}

\begin{proof}
The nonsingularity follows from Proposition~\ref{prop:non-singular}.  In order to prove that the pairing is hermitian, we will investigate the following diagram.
\[\xymatrix@R+0.5cm @C+0.5cm{0 \ar[r] & C_*(A;V) \ar[r] & C_*(N;V) \ar[r]^{q_A} & C_*(N,A;V) \ar[r] & 0 \\
0 \ar[r] &  C_*(B;V) \ar[r] \ar[u]_{\simeq} & C_*(N;V) \ar[r]^{q_B} \ar[u]^-{\mathcal{S}_*}_-{\simeq} \ar[ur]^{q_A} & C_*(N,B;V) \ar[r] \ar @{-->}[u]_{\sigma} & 0
 }\]
We will construct a chain equivalence $C_*(B;V) \to C_*(A;V)$ so that that the left hand square commutes up to homotopy. There is therefore an induced homotopy equivalence~$\sigma$ between the mapping cones, unique up to homotopy, shown as a dashed arrow in the diagram, that makes the right hand square commute up to homotopy.  Since the diffeomorphism $\mathcal{S}$ is isotopic to the identity, the upper triangle commutes up to homotopy.  Thus the lower right triangle also commutes up to homotopy.  But this is exactly the requirement for a hermitian triad in Definition~\ref{defn:hermitian-triad}.

We can and will assume that $\partial N = S^1 \times S^1$ has one connected component, for simplicity of exposition.  The argument for the general case just requires duplicating the following argument for each copy of $S^1 \times S^1$ in $\partial N$.

Let $A \cap B =D$, which is homeomorphic to $S^1 \times S^0$.
The choices of parametrisation determine a diffeomorphism $B \toiso A$, such that the diagram
\[\xymatrix{A \ar[r] & N \\ B \ar[r] \ar[u] & N \ar[u]_{\mathcal{S}} }\]
commutes.
We need to carefully consider the basepoints and basing paths so that the corresponding diagram for chain complexes with $V$ coefficients commutes.  First we consider $\Z[\pi]$ coefficients, and then at the end we will specialise to $V$ coefficients.

Choose a path $\gamma_0$ from the basepoint of $N$ to the point $(1,-1)$ of $D=S^1 \times S^0$.  For the basepoint of $A$ use $(1,-1) \in S^1 \times S^0$, and in order to define the chain complex $C_*(A;\Z[\pi])$ use the path $\gamma_0$.
Divide the longitude $\ell$ into $\ell = \ell_A \cup \ell_B$, where $\ell_A = \ell \cap A$ and $\ell_B = \ell \cap B$.  For the basepoint of $B$, use $(1,1) \in D = S^1 \times S^0$, and for the basing path use the concatenation $\ell_A \cdot \gamma_0$.  Note that $\mathcal{S}$ sends $\ell_A \cdot \gamma_0$ to $\ell \cdot \gamma_0$.

With these basing paths, we have representatives $C_*(A;\Z[\pi]) = C_*(B;\Z[\pi]) = (\Z[\pi] \xrightarrow{(g-1)} \Z[\pi])$ for the chain equivalence classes of the chain complexes of $A$ and $B$, where $g \in \pi$ is the homotopy class of curve $\gamma_0 \cdot \mu_{-1} \cdot \overline{\gamma_0}$ corresponding to the meridian $\mu_{-1}$ of $\partial N$ that traverses $S^1 \times \{-1\} \subset D$ once with positive orientation.  In particular, one should observe that there is a homotopy $\gamma_0 \cdot \mu_{-1} \cdot \overline{\gamma_0} \sim \gamma_0 \cdot \ell_A \cdot \mu_{-1} \cdot \overline{\ell_A} \cdot \overline{\gamma_0}$.
We can therefore use the identity map for the left hand vertical map of the diagram above.

We remark that in the symmetric Poincar\'{e} triad, the maps $C_*(D;\Z[\pi]) \to C_*(A;\Z[\pi])$ and $C_*(D;\Z[\pi]) \to C_*(B;\Z[\pi])$ are required to be split injections, so these chain complexes for $A$ and $B$ are only chain equivalent to the chain complexes we need to use to compute the Blanchfield form.  However for the current proof we only need the chain equivalence classes, since we only need the left hand square of the diagram above to commute up to homotopy.

The chain complex $\mathcal{S}_*(C_*(B;\Z[\pi]))$ is given by $(\Z[\pi] \xrightarrow{(\ell g \ell^{-1} -1 )} \Z[\pi])$, whereas the image of $C_*(A;\Z[\pi])$ is the chain complex $(\Z[\pi] \xrightarrow{(g-1)} \Z[\pi])$ as discussed above.  
Since $g$ and $\ell$ form a basis of $\pi_1(\partial N) = \Z \oplus \Z \leq \pi_1(N)$, they commute, and so $\ell g \ell^{-1} = g$.  It follows that the left hand square of the diagram above commutes.
Therefore, as explained at the beginning of the proof, we obtain the chain equivalence $\sigma \colon C_*(N,B;V) \simeq C_*(N,A;V)$ as required.

\end{proof}

\bibliographystyle{alpha}
\def\MR#1{}
\bibliography{research}

\end{document}